\documentclass[12pt]{amsart}
\usepackage{amssymb,amscd}
\usepackage[colorlinks=true,allcolors = blue]{hyperref} 

\def\>{\relax\ifmmode\mskip.666667\thinmuskip\relax\else\kern.111111em\fi}
\def\<{\relax\ifmmode\mskip-.333333\thinmuskip\relax\else\kern-.0555556em\fi}
\def\vsk#1>{\vskip#1\baselineskip}
\def\vv#1>{\vadjust{\vsk#1>}\ignorespaces}
\def\vvn#1>{\vadjust{\nobreak\vsk#1>\nobreak}\ignorespaces}
\def\vvgood{\vadjust{\penalty-500}} \let\alb\allowbreak
\def\fratop{\genfrac{}{}{0pt}1}
\def\satop#1#2{\fratop{\scriptstyle#1}{\scriptstyle#2}}
  \let\ssize\scriptstyle
\let\sssize\scriptscriptstyle

\let\Medskip\medskip
\def\medskip{\par\Medskip}
\let\Bigskip\bigskip
\def\bigskip{\par\Bigskip}

\let\Maketitle\maketitle
\def\maketitle{\Maketitle\thispagestyle{empty}\let\maketitle\empty}

\newtheorem{thm}{Theorem}[section]
\newtheorem{cor}[thm]{Corollary}
\newtheorem{lem}[thm]{Lemma}

\newtheorem{defn}[thm]{Definition}

\numberwithin{equation}{section}

\theoremstyle{definition}

\let\mc\mathcal
\let\nc\newcommand

\let\dl\delta
\let\Dl\Delta
\let\epe\epsilon

\let\la\lambda

\let\phi\varphi

\let\Ups\Upsilon
\let\om\omega

\let\der\partial

\let\bra\langle
\let\ket\rangle

\let\ge\geqslant
\let\geq\geqslant
\let\le\leqslant
\let\leq\leqslant

\let\on\operatorname
\let\bi\bibitem
\let\bs\boldsymbol

\def\C{{\mathbb C}}
\def\Z{{\mathbb Z}}
\def\R{{\mathbb R}}

\def\F{{\mc F}}

\def\Ic{{\mc I}}

\def\Oc{{\mc O}}

\def\+#1{^{\{#1\}}}
\def\lsym#1{#1\alb\dots\relax#1\alb}
\def\lc{\lsym,}

\def\Re{\on{Re}}
\def\Res{\on{Res}}

\def\cirs{{\raise.2ex\hbox{$\sssize\circ$}}}

\def\beq{\begin{equation}}
\def\eeq{\end{equation}}
\def\be{\begin{equation*}}
\def\ee{\end{equation*}}

\nc{\bea}{\begin{eqnarray*}}
\nc{\eea}{\end{eqnarray*}}
\nc{\bean}{\begin{eqnarray}}
\nc{\eean}{\end{eqnarray}}
\nc{\Ref}[1]{{\rm(\ref{#1})}}

\let\ga\gamma
\let\Ga\Gamma

\nc{\bla}{{\bs\la}}
\nc{\Il}{{\Ic_{\bla}}}
\nc{\Fla}{\F_\bla}
\nc{\tfl}{{T^*\<\Fla}}
\nc{\GL}{{GL_n(\C)}}
\nc{\GLC}{{GL_n(\C)\times\C^*}}

\def\pii{\pi\sqrt{\<-1}}

\let\Dx D

\def\zzz{z_1\lc z_n}

\let\sd s 

\def\zb{\bs z}

\def\ddk_#1{q_{#1}\<\>\frac\der{\der\<\>q_{#1}}}

\def\bul{\mathbin{\raise.2ex\hbox{$\sssize\bullet$}}}
\def\intt{\mathchoice
{\mathop{\raise.2ex\rlap{$\,\,\ssize\backslash$}{\intop}}\nolimits}
{\mathop{\raise.3ex\rlap{$\,\sssize\backslash$}{\intop}}\nolimits}
{\mathop{\raise.1ex\rlap{$\sssize\>\backslash$}{\intop}}\nolimits}
{\mathop{\rlap{$\sssize\<\>\backslash$}{\intop}}\nolimits}}

\def\GZ/{Gelfand-Zetlin}
\def\KZ/{{\slshape KZ\/}}
\def\qKZ/{{\slshape qKZ\/}}
\def\qKZB/{{\slshape qKZB\/}}
\def\XXX/{{\slshape XXX\/}}
\def\XXZ/{{\slshape XXZ\/}}

\def\zz{{\bs z}}

\def\qq{{\bs q}}

\def\yy{{\bs y}}

\def\ZZ{{\bs Z}}
\def\yy{{\bs y}}

\def\pti{\textit{pt}}

\def\Tdd{\acute T}
\def\Zdd{\acute Z}
\def\8{{\infty}}

\def\Der{\mathrm{Der}}
\def\Coh{\mathrm{Coh}}

\begin{document}

\hrule width0pt
\vsk->

\title[Quantum differential equation and K-theory for a projective space]
{Equivariant quantum differential equation, Stokes bases, and K-theory for a projective space}

\author[Vitaly Tarasov and Alexander Varchenko]
{Vitaly Tarasov$\>^\circ$ and Alexander Varchenko$\>^\star$}

\maketitle

\begin{center}
{\it $^{\star}\<$Department of Mathematics, University
of North Carolina at Chapel Hill\\ Chapel Hill, NC 27599-3250, USA\/}

\vsk.5>
{\it $^{\star}\<$Faculty of Mathematics and Mechanics, Lomonosov Moscow State
University\\ Leninskiye Gory 1, 119991 Moscow GSP-1, Russia\/}

\vsk.5>
{\it $\kern-.4em^\circ\<$Department of Mathematical Sciences,
Indiana University\,--\>Purdue University Indianapolis\kern-.4em\\
402 North Blackford St, Indianapolis, IN 46202-3216, USA\/}

\vsk.5>
{\it $^\circ\<$St.\,Petersburg Branch of Steklov Mathematical Institute\\
Fontanka 27, St.\,Petersburg, 191023, Russia\/}
\end{center}

{\let\thefootnote\relax
\footnotetext{\vsk-.8>\noindent
$^\circ\<${\sl E\>-mail}:\enspace vt@math.iupui.edu\>, vt@pdmi.ras.ru\>,
supported in part by Simons Foundation grant 430235
\\
$^\star\<${\sl E\>-mail}:\enspace anv@email.unc.edu\>,
supported in part by NSF grant DMS-1665239}}

\begin{abstract}
We consider the equivariant quantum differential equation for the projective space $P^{n-1}$.
We prove an equivariant gamma theorem for $P^{n-1}$,
which describes the asymptotics of the differential equation at its regular singular point in terms of the
equivariant characteristic gamma class of the tangent bundle of $P^{n-1}$. We describe the Stokes bases
of the differential equation at its irregular singular point in terms of the exceptional bases
of the equivariant K-theory algebra of $P^{n-1}$ and a suitable braid group action on the
set of exceptional bases.

Our results are an equivariant version of the well-know results of B.\,Dubrovin
and D.\,Gu\-zzetti.
\end{abstract}

\vsk>
{\leftskip3pc \rightskip\leftskip \Small
\setbox0\hbox{{\it Key words\/}: {}}\parindent\wd 0
\hangindent\parindent\noindent\box0
Equivariant quantum differential equation,
equivariant $K$-theory, $q$-hyper\-geo\-met\-ric solutions, braid group action
\vsk.5>
\noindent
{\it 2010 Mathematics Subject Classification\/}: 14N35, 34M40, 17B80
\par}

\vsk>

\begin{center}
To Boris Dubrovin with admiration
\end{center}

\bigskip

{\small\tableofcontents\par}

\setcounter{footnote}{0}
\renewcommand{\thefootnote}{\arabic{footnote}}

\section{Introduction}

The quantum differential equation of the complex projective space $P^{n-1}$
is an ordinary differential equation
\beq
\label{Qde}
\Bigl(p\frac{d}{dp} - x *_{p}\Bigr) I(p) =0,
\eeq
where the unknown function $I(p)$ takes values in
the cohomology algebra $H^*(P^{n-1};\C)$ and
$x *_p: H^*(P^{n-1};\C) \to H^*(P^{n-1};\C)$ is the
operator of {\it quantum} multiplication by the first Chern class of the tautological line bundle over $P^{n-1}$.
The differential equation has two singular points: a regular singular point at $p=0$ and an irregular singular point
at $p=\infty$.
The quantum differential equation has the following remarkable structures.

\vsk.2>
The specially normalized asymptotics of its solutions
at $p=0$ can be described in terms of the characteristic
gamma class of the tangent bundle of $P^{n-1}$. This description
of the asymptotics by B.\,Dubrovin in \cite{D1} was the first example of a gamma theorem,
which is proved now in many examples and is
known as the gamma conjecture, see \cite{D1, D2, KKP, GGI, GI, GZ, CDG}.

The Stokes matrices of the quantum differential equation at the irregular singular point $p=\infty$ are described in
terms of the braid group action on the set of full collections of exceptional objects in the
derived category $\,\Der^b(\Coh(P^{n-1}))$ of coherent sheaves on $P^{n-1}$.
That phenomenon was predicted by B.\,Dubrovin in \cite{D1} for Fano varieties and was proved
for $P^{n-1}$ by D.\,Guzzetti \cite{Gu}. That braid group action is described with the help of a
certain non-symmetric bilinear from on the K-theory algebra $K(P^{n-1},\C)$.

\vsk.2>

In this paper we consider the {\it equivariant} quantum differential equation of the projective space $P^{n-1}$ and
establish similar results.

\vsk.2>
In the equivariant case the torus $T=(\C^\times)^n$
acts on $P^{n-1}$ and the quantum differential equation takes the form
\bean
\label{QDe}
\Bigl(p\frac{d}{dp} - x *_{p,\zz}\Bigr) I(p, z_1\lc z_n) =0,
\eean
where $\zz=(\zzz)$ are equivariant parameters.
We also have
a system of the \qKZ/ difference equations
\beq
\label{QKZ}
I(p,z_1, \dots, z_i-1, \dots, z_n ) = K_i(p,\zzz) I(p,\zzz),
\qquad i=1\lc n\,,\kern-.6em
\eeq
where $ K_i(p,\zzz)$ are suitable linear operators.
The joint system of the equivariant quantum differential equation and \qKZ/
difference equations is compatible. The space of solutions of this system
is a module over the ring of scalar functions in $\zzz$, \,1-periodic with
respect to each of the variables $\zzz$.

\vsk.2>
We prove an equivariant gamma theorem, which describes the asymptotics of
solutions at $p=0$ of the equivariant quantum differential equation in terms of the
equivariant characteristic gamma-class of the tangent bundle of $P^{n-1}$,
see Theorem \ref{thm gth}.

\vsk.2>

We describe the Stokes bases
of the equivariant quantum differential equation at $p=\infty$.
For that we identify the space of solutions of the joint system of equations \Ref{QDe} and \Ref{QKZ} with
the space of the equivariant K-theory algebra $K_T(P^{n-1},\C)$. We introduce a sesquilinear form
on $K_T(P^{n-1},\C)$, exceptional bases of $K_T(P^{n-1},\C)$, a braid group action on the exceptional bases,
and describe the Stokes bases in terms of that braid group action, see Theorem \ref{thm main}.

\vsk.2>
To prove these results we use integral representations for solutions of the joint system
of equations \Ref{QDe} and \Ref{QKZ} obtained in \cite{TV6}.
In \cite{TV6} we constructed $q$-hypergeometric integral representations for
solutions of the joint systems of equivariant quantum differential equations and associated \qKZ/ difference
equations for the cotangent bundle $\tfl$ of a partial flag variety $\mc F_\bla$. In
a suitable limit those solutions become solutions of the corresponding equations for the
partial flag variety $\mc F_\bla$. In this paper we use the special case of $\mc F_\bla=P^{n-1}$.

\vsk.2>

The important role in this paper is played by the identification of the space of solutions of the joint system
of equations \Ref{QDe} and \Ref{QKZ} with the space of the K-theory algebra $K_T(P^{n-1},\C)$. This identification
also comes from \cite{TV6}. Earlier examples of such an identification see in \cite{TV3,TV7}.

\vsk.2>
We would like to stress that the equivariant case is simpler than the corresponding non-equivariant case.
The equivariant case is more rigid because, in addition to the quantum differential equation,
we also have the compatible system of difference equations, and therefore there are less problems with choices of
normalizations of solutions.

\vsk.2>

The paper is organized as follows. In Section \ref{sec prs} we introduce the equivariant
cohomology and K-theory algebra of $P^{n-1}$.
In Section \ref{sQde} we introduce the equivariant quantum differential equation and \qKZ/ difference equations.
In Section \ref{sec sol} we describe the integral representations for solutions and asymptotics of solutions
at $p=0$.
In Section \ref{sec Ainf} we discuss asymptotics of solutions at $p=\infty$, introduce Stokes bases in the space of solutions.
In Section \ref{sec ebbg} we introduce exceptional bases in the space of solutions and a braid group action on the set of exceptional bases. In Section \ref{sec StB} we describe the Stokes bases of the equivariant quantum differential equation
at $p=\infty$. Our proofs in Section \ref{sec StB} are similar to the corresponding proofs in \cite{Gu}.

\medskip
The authors thank Giodano Cotti, Alexander Givental, and Richard Rim\'anyi
for many helpful discussions.

\section {Projective space}
\label{sec prs}

\subsection{Equivariant cohomology}
\label{sec:equiv}

For \,$n\ge 2$\,, let $P^{n-1}$ be the projective space parametrizing
one-dimen\-si\-onal subspaces \,$F\subset\C^n$.

Let $\{u_1\lc u_n\}$ be the standard basis of $\C^n$. For \,$I\in\{1\lc n\}$\,,
let \,$\pti_I\in P^{n-1}$ \,be the point corresponding to the coordinate line
spanned by $u_I$.
The complex torus $T^n$ acts diagonally on $\C^n$, and hence on $P^{n-1}$.
The points $\pti_I$, $I=1\lc n$, compose the fixed point set.

We consider the equivariant cohomology algebra \,$H_{T^n}(P^{n-1};\C)$\,.
Denote by $x$ the equivariant Chern root of the tautological line bundle
$\mc L$ over \,$P^{n-1}$ \,with fiber \,$F$\,.
Denote by $\yy =(y_1\lc y_{n-1})$ the equivariant Chern
roots of the vector bundle over \,$P^{n-1}$ \,with fiber \,$\C^n/F$\,.
Denote by \,$\zb=(\zzz)$
\,the Chern roots corresponding to the factors of the torus \,$T^n$\,. Then
\begin{align}
\label{Hrel}
H_{T^n}(P^{n-1};\C)\,&{}=\,
\C[x,\zz]\>\Big/\Bigl\bra\,\prod_{a=1}^n\,(x-z_a)\Bigr\ket\,
\\
\notag
&{}=\,\C[x, \yy,\zz]^{S_{n-1}}\>\Big/
\Bigl\bra(u-x)\prod_{j=1}^{n-1}(u-y_j)-\prod_{a=1}^n\,(u-z_a)\Bigr\ket\,,
\end{align}
where $\C[x, \yy,\zz]^{S_{n-1}}$ is the algebra of polynomials in $x,\yy,\zz$ symmetric in the variables
$y_1\lc y_{n-1}$, \,and the isomorphism of the first quotient to the second one
sends an element $f(x,\zz)$ of the first quotient to the element $f(x,\zz)$ of
the second.

\vsk.2>
The cohomology algebra \,$H_{T^n}(P^{n-1};\C)$ \,is a module over
\,$H^*_{T^n}({pt};\C)=\C[\zz]$\,.

\subsection{Symmetric functions}
Consider the algebra $\C[\bs Z^{\pm 1}]=\C[Z_1^{\pm1}\lc Z^{\pm1}_n]$ of
Laurent polynomials and its elements
\bean
\label{sm}
s_k(\ZZ) &=&\sum_{1\leq i_1<\dots<i_k\leq n}Z_{i_1}\dots Z_{i_k},
\qquad k=1\lc n,
\\
\notag
m_k(\ZZ) &=& \sum_{\satop{i_1\ge0\lc\<\>i_n\ge0\!}{i_1\lsym+\>i_n=\>k}}
Z_1^{i_1}\dots Z_n^{i_n}, \qquad k\in\Z_{>0}\,.
\eean
Put $s_0=1$, $m_0=1$. Then
\bean
\label{ms}
\sum_{i=0}^k\,(-1)^i\,m_i(\ZZ)\,s_{k-i}(\ZZ)\,=\,0\,, \qquad k\in\Z_{>0}\,.
\eean

For $f(Z_1\lc Z_n)\in \C[\bs Z^{\pm1}]$ denote $f(\bs Z^{-1})=f(Z_1^{-1}\lc Z_n^{-1})$.

\subsection{Equivariant K-theory}
\label{sec:equK}

Consider the equivariant K-theory algebra
$K_{T^n}(P^{n-1};\C)$. We have
\beq
\label{HrelK}
K_{T^n}(P^{n-1},\C)\,=\,\C[X^{\pm1}\!,\ZZ^{\pm1}]\>\Big/\Bigl\bra
\,\prod_{a=1}^n\,(X-Z_a)\Bigr\ket\,.
\vv.2>
\eeq
Here the variable $X$ corresponds to the tautological line bundle
$\mc L$ over $P^{n-1}$;
the variables $Z_1\lc Z_n$ are the equivariant parameters corresponding to the factors of $T^n$;
$\C[X^{\pm 1},\ZZ^{\pm1}]$ is the algebra of Laurent polynomials in $X$, $Z_1\lc Z_n$.

\vsk.2>
The algebra \,$K_{T^n}(P^{n-1},\C)$ \>is a module over \,$K_{T^n}(\pti\>;\C)=\C[\ZZ^{\pm1}]$.

\vsk.2>
We have a map
\bea
\rho\ :\ K_{T^n}(P^{n-1},\C)\to K_{T^n}(P^{n-1},\C), \quad f(X,\ZZ) \mapsto f(X^{-1},\ZZ^{-1}),
\eea
which sends the class of a vector bundle to the class of the dual vector bundle.

\vsk.2>
The map to a point $\psi :P^{n-1}\to \pti$ gives us the push-forward map
$\psi_* : K_{T^n}(P^{n-1},\C) \to \C[\ZZ^{\pm1}]$ \,defined by the formula
\beq
\label{pfw}
\psi_* f(X,\ZZ)\,=\,
\sum_{a=1}^n \,\frac{f(Z_a, \ZZ)}{\prod_{j\ne a}(1-Z_a/Z_j)}\,
=\, - \sum_{a=1}^n \Res_{X=Z_a} \frac{f(X,\ZZ)}{X \prod_{j=1}^n\,(1-X/Z_a)}
\eeq
The push-forward map $\psi_*$ gives
us a symmetric bilinear form on $K_{T^n}(P^{n-1},\C) $ defined by the formula
$(f, g) = \psi_*(f g)$. We are interested in its non-symmetric
sesquilinear version,
\begin{align}
\label{sqK}
A(f,g)\,&{}=\,
\psi_*(X^n\rho(f)g)\, \frac{(-1)^{n-1}}{\prod_{j=1}^n Z_j}
\\[4pt]
\notag
&{}=\,\sum_{a=1}^n\,\frac{f(Z_a^{-1},\ZZ^{-1})\,g(Z_a,\ZZ)}
{\prod_{j\ne a}(1-Z_j/Z_a)}
\,=\,
\sum_{a=1}^n \Res_{X=Z_a}
\frac{f(X^{-1},\ZZ^{-1})\,g(X,\ZZ)}{X \prod_{a=1}^n\,(1-Z_a/X)}.
\end{align}

\begin{lem}
\label{lem Kfo}
For \,$i,j\in\Z$, we have \,$A(X^i,X^j)=m_{j-i}(\ZZ)$ if \,$i\le j$\,,
and \,$A(X^i,X^j)=0$ if \,$j<i<j+n$\,.
\qed
\end{lem}

\section{Quantum equivariant differential equation and
\\
\qKZ/ difference equations}
\label{sQde}

\subsection {Quantum multiplication}

In enumerative geometry the multiplication in the equivariant cohomology algebra
$H_{T^n}(P^{n-1},\C)$ is deformed. The deformed
{\it quantum} multiplication depends on the {\it quantum} parameter $p$ and equivariant parameters $\zz$.
The quantum multiplication is
determined by the $\C[\zz]$-linear operator
\bea
x *_{p,\zz} : H_{T^n}(P^{n-1},\C) \to H_{T^n}(P^{n-1},\C)
\eea
of multiplication by the generator $x\in H_{T^n}(P^{n-1},\C)$. In the basis $\{1, x, \dots,x^{n-1}\}$, we have
\begin{align*}
x *_{p,\zz}x^j\>&{}=\,x^{j+1}\>, \qquad j=0\lc n-2\,,
\\
x *_{p,\zz}x^{n-1}\>&{}=\,p+x^n\>=\,
p+\sum_{i=1}^n(-1)^{i-1}s_i(\zz)\,x^{n-i}\>,
\end{align*}
where $s_i(\zz)$ are the elementary symmetric functions in $\zz$\,.$^{(1)}$

{\let\thefootnote\relax
\footnotetext{\vsk-.8>\noindent
$^{(1)}$ These formulas were explained to us by A.\,Givental}}

We also use the basis $\{g_1\lc g_n\}$\,,
\be
g_i =\prod_{a=i+1}^n (x-z_a),
\quad i=1\lc n-1\,, \quad\on{and}\quad
g_n=\>1.
\ee
In this basis we have
\begin{align*}
x*_{p,\zz} g_i\>&{}=\,z_i g_i + g_{i-1}\,, \qquad i=2\lc n\,,
\\
x*_{p,\zz} g_1\>&{}=\,z_1 g_1 + p g_n\,.
\end{align*}

\subsection{R-matrices and \qKZ/ operators}
For $a,b\in\{1\lc n\}$, $a\ne b$, define a $\C[\zz]$-linear operator
\bea
R_{ab}(u)\ :\ H_{T^n}(P^{n-1},\C)\ \to\ H_{T^n}(P^{n-1},\C),
\eea
depending on $u\in\C$ and
called the {\it R-matrix}, by the formula
\begin{alignat*}2
R_{ab}(u)\, g_i\,&{}=\,g_i\,, && i \ne a,b\,,
\\
R_{ab}(u)\,g_b\,&{}=\,g_a\,,\qquad && R_{ab}(u)\, g_a\,=\,g_b + u g_a\,.
\end{alignat*}
These $R$-matrices satisfy the Yang-Baxter equation
\be
R_{ab}(u-v) R_{ac}(u) R_{bc}(v)\,=\,R_{bc}(v) R_{ac}(u) R_{ab}(u-v),
\ee
for all distinct $a,b,c$, and the inversion relation
\be
R_{ab}(u) R_{ba}(-u)\,=\,1\,.
\ee
Define the operators $E_1,\dots E_n$ such that
\bea
E_i\, g_j \,=\, \dl_{ij}\,g_i.
\eea
Define the \qKZ/ operators $K_1\lc K_n$ by the formula
\bea
K_i = R_{i,i-1}(z_i-z_{i-1}-1) \dots R_{i,1}(z_i-z_1-1) p^{-E_i} R_{i,n}(z_i-z_n) \dots R_{i,i+1}(z_i-z_{i+1})
\eea

\subsection{Quantum differential equation and \qKZ/ difference equations}

The {\it equivariant quantum differential equation} is the differential equation
\bean
\label{qde}
\Bigl(p\frac{d}{dp} - x *_{p,\zz}\Bigr) I(p,z_1\lc z_n) =0.
\eean
The system of the {\it \qKZ/ difference equations} is the system of difference equations
\beq
\label{qKZ}
I(p,z_1, \dots, z_i-1, \dots, z_n ) = K_i(p,\zzz) I(p,\zzz),
\qquad i=1\lc n\,.\kern-1em
\eeq

In these equations the unknown function $I(p,\zz)$ takes values in the cohomology
algebra $H_{T^n}(P^{n-1},\C) $ extended
by functions in $p,\zz$.

\begin{thm}
\label{thm compa} The joint system of equations \Ref{qde} and
\Ref{qKZ} is compatible.

\end{thm}

\begin{proof}
The proof is straightforward.
\end{proof}

\section{Integral representations for solutions}
\label{sec sol}

The quantum differential equation \Ref{qde} was solved by A.\,Givental in \cite{G1}. In this section
we follow \cite{TV6} and describe the integral representations for solutions
of the joint system of equations \Ref{qde} and \Ref{qKZ}.

\vsk.2>
Notice that the space of solutions to the joint system of equations \Ref{qde}
and \Ref{qKZ} is a module over the ring of scalar functions in $\zzz$,
\,1-periodic with respect to each of the variables $\zzz$.

\subsection{Master and weight functions}

Consider the variables $t, p$, $\zz=(\zzz)$.
Define the {\it master function} $\Phi$ and
$H_{T^n}(P^{n-1};\C)\<\>$-valued {\it weight function} $W$ by the formulas:
\vvn.3>
\begin{align}
\label{PHI}
\Phi(t, p,\zz)\,=\,(\>e^{\<\>\pii\,(2-n)}\>p\>)^t\,
\prod_{a=1}^n\,\Ga(z_a\<-t)\,,\qquad
W(t,\yy)\,=\,\prod_{j=1}^{n-1}\,(y_j\<-t)\,,
\end{align}
where \,$\Ga$ is the gamma function.

\subsection{Solutions as Jackson integrals}
\label{seJ}

Consider $\C$ with coordinate $p$ and $\C^n$ with coordinates $\zz=(z_1\lc z_n)$.

Let $L'$ be the $p$-line $\C$ with a cut to fix the argument of $p$, that is, we delete
from $\C$ a ray from $0$ to $ \infty$ and fix the argument of $p$ on the complement.

Let \,$L''$ \>be the complement in \,$\C^n$ \,to the union of
the hyperplanes
\beq
\label{zzh}
z_a\<-z_b\>=\>m \quad
\on{for\, all}\ a,b=1\lc n,\ \,a\ne b\,, \on{and\, all} \,m\in\Z\,.
\vv.3>
\eeq
Set
$L=L'\!\times L''\!\subset\<\C \times\C^n$.
For $J=1\lc n$ define
\vvn.3>
\beq
\label{Pshh}
\Psi_J(p, \yy,\zz)\,=\,-\sum_{r\in\Z_{\ge0}}
\Res_{t=z_J+r}\Phi(t,p,\zz)\,W(t,\yy)\,.
\eeq
These sums are called the {\it Jackson integrals}.

\begin{thm} [\cite{TV6}]
\label{thm lsol}
The functions $\Psi_J(p, \yy,\zz)$, $J=1\lc n$,
belong to
the extension of \,$H_{T^n}(P^{n-1},\C)$ \,by holomorphic functions in \,$p,\zz$
\,on the domain $L\subset \C\times \C^n$.
Each of the functions
is a solution to the joint system of equations \Ref{qde} and \Ref{qKZ}.
These functions form a basis of solutions.
\end{thm}

\begin{proof} The theorem is proved in \cite[Section 11.4]{TV6}, see formula (11.18) in there.
In particular, the fact that the functions form a basis follows from the determinant formula (11.23).

In fact, in Section 11.4
the solutions to the joint system of the equivariant quantum differential equations and associated \qKZ/ equations
are described
for an arbitrary partial flag variety.
\end{proof}

The solutions $\Psi_J(p, \yy,\zz)$, $J=1\lc n$, are called the {\it $q$-hypergeometric} solutions.

\subsection{Asymptotics $p\to 0$ and equivariant gamma theorem}

\begin{cor}[{\cite[Formula (11.19)]{TV6}}]
\label{cor asy}
As $p\to 0$, we have
\beq
\label{asy}
\Psi_J(p, \yy,\zz)\,=\,
(e^{\<\>\pii\,(2-n)}\> p\>)^{z_J}\prod_{a\ne J}^n\<\Ga(1+z_a-z_J)
\Bigl(\Dl_J + \sum_{k=1}^\infty p^k\Psi_{J,k}(\yy,\zz)\Bigr),
\eeq
where the equivariant class $\Dl_J$ restricts to $1$ at the fixed point $\pti_J$
and restricts to zero at all other fixed points $\pti_I$ with $I\ne J$. The
classes $\Psi_{J,k}(\yy,\zz)$
are suitable rational functions in \,$\zz$ regular on $L''$.
\qed
\end{cor}

Recall that $\prod_{j=1}^{n-1}(y_j-x) \in H_{T^n}(P^{n-1},\C)$ is
the equivariant total Chern class of the tangent bundle of $P^{n-1}$ and
$x \in H_{T^n}(P^{n-1},\C)$ is the equivariant first Chern class $c_1(\mc L)$
of the tautological line bundle $\mc L$ over $P^{n-1}$.
The function $\hat\Ga_{P^{n-1}}=\prod_{a=1}^{n-1}\Ga(1+ y_a-x)$ is called
the {\it equivariant gamma class} of the tangent bundle of $P^{n-1}$.
Corollary \ref{cor asy} can be reformulated as the following statement.

\begin{thm}
\label{thm gth}
The leading term of the asymptotics as $p\to 0$ of the \,$q$-hypergeometric
solutions $(\Psi_J(p, \yy,\zz))_{J=1}^n$ is the product of the equivariant
gamma class of the tangent bundle of \,$P^{n-1}\!$ and the exponential of
the equivariant first Chern class of the tautological line bundle $\mc L$:
\vvn.1>
\beq
\label{FHEo}
\bigl(\>e^{\<\>\pii\,(2-\<\>n)}\>p\bigr)^{c_1(\mc L)}\,
\hat\Ga_{P^{n-1}}.
\vv->
\eeq
\qed

\end{thm}

This assertion is an equivariant analog of Dubrovin's gamma theorem for $P^{n-1}$,
see \cite{D1, D2} and also \cite{KKP, GGI, GI, GZ, CDG}.

\subsection{Solutions as elements of the equivariant \,$K\<$-theory}
\label{equivkth}

Introduce new functions:
\bean
\label{Zdd}
\Tdd =e^{2\pii\, t}, \quad
\Zdd_j=e^{2\pii \,z_j},
\quad j=1\lc n.
\eean
Denote ${\bs \Zdd}=(\Zdd_1\lc\Zdd_n)$.

\vsk.2>
Let \,$Q(X, \ZZ)\in\C[X^{\pm1},\ZZ^{\pm1}]$ \,be a Laurent polynomial.
Define
\beq
\label{PPI}
\Psi_Q(p, \yy, \zz)\,=\,
\sum_{J=1}^n\,Q(\Zdd_J,{\bs \Zdd})\,\Psi_J(p, \yy,\zz).
\vv-.3>
\eeq

Clearly, $\Psi_Q(p, \yy, \zz)$ is a solution on the domain $L\subset \C\times \C^n$ of the joint system
\Ref{qde} and \Ref{qKZ},
as a linear combination of solutions
$\Psi_J(p, \yy,\zz)$ with coefficients, 1-periodic with respect to $\zzz$ and independent of $p$.

\smallskip

It is also clear that if $Q$ lies in the ideal in $\C[X^{\pm1},\ZZ^{\pm1}]$
generated by the polynomial $\prod_{a=1}^n\,(X-Z_a)$, then
$\Psi_Q(p, \yy, \zz)$ is the zero solution. Hence formula \Ref{PPI} defines
a map
$Q\mapsto \Psi_Q(p, \yy, \zz)$ from
the equivariant K-theory algebra $K_{T^n}(P^{n-1},\C)$ to the space of solutions on
the domain $L$ to the joint system \Ref{qde} and \Ref{qKZ}.

\subsection{Solutions $\Psi^m$}
For $m\in \Z$\,, denote by $\Psi^m(p, \yy, \zz)$ the solution $\Psi_Q(p, \yy, \zz)$ corresponding to the Laurent polynomial $Q=X^{m-1}$.

\begin{cor}
\label{cor z-rel}
For any $k\in \Z$\,, we have
\bean
\label{Psi reln}
\sum_{i=0}^n (-1)^{n-i} s_{n-i}({\bs \Zdd}) \,\Psi^{k+i}(p, \yy, \zz)\ =\ 0,
\eean
where \,$s_0({\bs\Zdd})\lc s_n({\bs\Zdd})$ are the elementary symmetric functions in ${\bs \Zdd}$.
\qed
\end{cor}

\begin{thm}
[{\cite[Theorem 11.3]{TV6}}]
\label{thm onto} For any $k\in \Z$, the solutions $\Psi^{k+m}(p, \yy, \zz)$,
$m=0\lc n-1$, form a basis of the space of solutions on the domain $L$ of the joint system
\Ref{qde} and \Ref{qKZ}.
\end{thm}

\subsection{Module $\mc S_n$}

The space of solutions of the joint system of equations \Ref{qde} and \Ref{qKZ} is a
module over the algebra of functions in $z_1\lc z_n$, which are
1-periodic with respect to each variable.

\vsk.2>
We will consider the space $\mc S_n$ of solutions of the form
\bean
\label{def sol}
\sum_{m=1}^n Q_m(\bs\Zdd)\,\Psi^{m}(p, \yy, \zz), \qquad \on{where}\ \
Q_m(\ZZ)\in \C[\ZZ^{\pm1}].
\eean
This space is a $\C[\ZZ^{\pm1}]$-module, in which multiplication by $Q(\ZZ)$ is
defined as multiplication
by $Q(\bs\Zdd)$.
With this choice of the space of solutions,
we allow ourselves to multiply solutions $\Psi^{m}(p, \yy, \zz)$ only
by \,1-periodic functions of the form $Q_m(\bs\Zdd)$, where $Q(\ZZ)\in
\C[\ZZ^{\pm1}]$.

\vsk.2>
By Corollary \ref{cor z-rel}, the module $\mc S_n$ contains all solutions
$\Psi^{m}(p, \yy, \zz)$, $m\in\Z$.

\begin{cor}
\label{cor KS}

The module $\mc S_n$ contains a basis of solutions of the joint system
\Ref{qde} and \Ref{qKZ}. Moreover, the map $\theta : K_{T^n}(P^{n-1},\C) \to
\mc S_n$, defined by
\bean
\label{theta}
\theta : X^{m-1} \mapsto \Psi^{m}(p, \yy, \zz), \qquad m\in \Z,
\eean
is an isomorphism of the $\C[\ZZ^{\pm1}]$-modules.
\end{cor}

\begin{proof}
The corollary follows from Theorem \ref{thm onto}.
\end{proof}

\vsk.2>

Using the isomorphism $\theta$ we define a sesquilinear form $A$
on $\mc S_n$ as the image of the form $A$ on $K_{T^n}(P^{n-1},\C)$.

\subsection{Monodromy of the quantum differential equation}
The equivariant quantum differential equation \Ref{qde} has two singular points. A regular singular point at
$p=0$ and an irregular singular point at $p=\infty$.

Fix $(p,\zz)$ and increase the argument of $p$ by $2\pi$. The analytic continuation of the solutions
along this curve will produce the
{\it monodromy}
operator $M(\zz)$ on the space of solutions.

\begin{thm}
\label{thm mdr}
For every $m\in \Z$ we have $M(\zz) : \Psi^{m}(p, \yy, \zz) \mapsto \Psi^{m+1}(p, \yy, \zz)$. In particular, for any $k\in\Z$,
the matrix of the monodromy operator in the basis $\{\Psi^{k+m}(p, \yy, \zz)\, | \, m=0\lc n-1\}$ is the companion matrix of the polynomial
$X^n-s_{1}({\bs \Zdd}) X^{n-1} + \dots +(-1)^ns_{n}({\bs \Zdd})$, that defines the relation in the equivariant K-theory algebra,
\bea
\left(
	\begin{matrix}
0 & 0 & \dots &\dots & 0 & (-1)^{n+1}s_{n}({\bs \Zdd})
\\
1 & 0 & \dots &\dots & 0 & (-1)^ns_{n-1}({\bs \Zdd}) \\
0 & 1 & \dots & \dots & \dots &\dots \\
	\dots & \dots & \dots & \dots & \dots & \dots \\
	0 & 0 & \dots&\dots &0 &-s_{2}({\bs \Zdd})\\
	0 & 0 &\dots & \dots &1 & s_{1}({\bs \Zdd})\\
	\end{matrix}
	\right).
	\eea

\end{thm}

\begin{proof}
The shift of the argument of $p$ by $2\pi$ leads to multiplication
by $e^{2\pii\, z_J}$ of each term in the sum in \Ref{Pshh}. This means
that $M(\zz) : \Psi_J(p, \yy,\zz) \mapsto e^{2\pii \,z_J}\Psi_J(p, \yy,\zz)$\,,
and hence
$M(\zz) : \Psi^{m}(p, \yy, \zz) \mapsto \Psi^{m+1}(p, \yy, \zz)$ for any $m\in\Z$.
Now the shape of the monodromy matrix in the basis
$\{\Psi^{k+m}(p, \yy, \zz)\, | \, m=0\lc n-1\}$ follows from relation \Ref{HrelK} in the K-theory algebra.
\end{proof}

\subsection{Solutions as integrals over a parabola}
For \,$A\in\C$\,, let \,$C(A)\subset\C$ \,be the parabola with the following
parametrization:
\vvn.3>
\beq
\label{CA}
C(A)\,=\,\{\bigl(A+s^2\<\<+s\>\sqrt{\<-1}\>\bigr)\ \,|\ \,s\in\R\,\}\,.
\vv.2>
\eeq
Given \,$\zz$, take \>$A$ \,such that all the points \,$\zzz$ lie
inside \,$C(A)$\,. The integral
\Ref{PshP} below does not depend on a particular choice of $A$\,,
so we will denote $C(A)$ by $C(\zz)$.

\begin{lem}
[{\cite[Lemma 11.5]{TV6}}]

\label{lem int}
For any Laurent polynomial \,$Q(X,\ZZ)$ we have
\beq
\label{PshP}
\Psi_Q(p,\yy,\zz) \,=\,\frac1{2\pii}\,
\int_{C(\zz)} Q(\Tdd;{\bs \Zdd})\,
\Phi(t, p,\zz)\,W(t,\yy)\, dt,
\eeq
where the integral converges for any $(p,\zz) \in L$.
\end{lem}

In particular, we have
\beq
\label{Pm}
\Psi^m(p,\yy,\zz) \,=\,\frac1{2\pii}\,
\int_{C(\zz)} e^{2\pii\,mt} e^{-\pii\,nt}p^t\,
\prod_{a=1}^n\,\Ga(z_a\<-t)\,\prod_{j=1}^{n-1}\,(y_j\<-t) \,dt\,.
\eeq

\section{Asymptotics as $p\to\8$}
\label{sec Ainf}

\subsection{Asymptotics of $\Psi^m$}
We make the change of variables:
\beq
\label{chv}
p=s^n,\qquad s=re^{-2\pii\,\phi}, \quad -s=re^{-\pii\>-\>2\pii\,\phi},
\quad r\geq 0\,,\ \;\phi\in\R\,.\kern-2em
\eeq
Denote \,$\om = e^{2\pii/n}$\,.

\begin{lem}
\label{lem Min}
For \,$m\in\Z$\,, \,$\phi\in\R$\,, \,and
\beq
\label{Min}
\frac mn - 1\,<\,\phi\,<\,\frac mn\;,
\vvgood
\eeq
we have the asymptotic expansion as \,$r\to\8$\,,
\beq
\label{exp}
\Psi^m(s^n\!,\yy,\zz)\,=\,
\frac{(2\pi)^{(n-1)/2}}{\sqrt n}\;e^{n\>\om^m\<s}\>
(-\om^m\<s)^{\sum_{a=1}^n z_a +(1-n)/2}\,
\prod_{j=1}^{n-1}\,(y_j\<-\om^m\<s)\,\bigl(1+\Oc(1/s)\bigr)\>.\kern-1em
\eeq
where \,$\arg(-\om^m\<s)=2\pi\>m/n-\pi-2\pi\>\phi$\,, \,so that
\,$|\arg(-\om^m\<s)|<\pi$\,.
\end{lem}

\begin{proof}
The proof of this lemma is a modification of the proof of \cite[Lemma 5]{Gu}.
Consider the logarithm of the integrand in \Ref{Pm},
\be
\Ups(t,\yy,\zz)\,=\,
\log\<\Bigl(e^{2\pii\,mt} e^{-\pii\,nt}p^t\,
\prod_{a=1}^n\,\Ga(z_a\<-t)\,\prod_{j=1}^{n-1}\,(y_j\<-t)\Bigr)\>,
\ee
and apply to \,$\Ups$ the Stirling formula
\bea
\log\,\Ga(u)\,=\,u\log u\>-\>u\>+\>\frac 12\log(2\pi/u) + \Oc(1/u)\,,
\quad \on{as}\ u\to\8\,, \ \ |\arg u| < \pi\,.
\eea
As \,$t\to\8$ and \,$|\arg (-t)|<\pi$, we have
\vvn.4>
\begin{align*}
\Ups\,&{}=\,
2\pii\,mt -\pii\,nt + n\>\bigl(\>\log r - 2\pii\,\phi\bigr)\,t
\\
&\>{}+\,\sum_{a=1}^n
\Bigl(\<(z_a - t)\log (z_a\<-t) - (z_a\<-t) +
\frac 12 \log\<\Bigl(\frac {2\pi}{z_a\<-t}\Bigr)\<\Bigr)+
\sum_{j=1}^{n-1}\,\log(y_j\<-t)+ \Oc(1/t)\>.
\end{align*}
The critical point equation of this expression with respect to \,$t$ yields
\vvn.4>
\beq
\label{cr m}
\log(-t)\,=\,2\pii\;\frac mn\,-\,\pii\,+\,\log r\,-\,2\pii\,\phi\,+\,
\Oc(1/t)\,.
\eeq
This implies that for
\vvn-.5>
\beq
\label{ineq}
-\>\pi<\>2\pi\frac mn -\pi -2\pi\phi\><\pi\,,
\eeq
the function \,$\Ups(t,\yy,\zz)$ \,has a critical point \,$t_m\in\C$\,,
with respect to \,$t$\,, such that
\beq
\label{cr tp}
\log(-t_m)\,=\,\log(-\om^m\<s)\>+\>\Oc(1/r)\,,
\eeq
where \,$\arg(-\om^m\<s)=2\pi\>m/n-\pi-2\pi\>\phi$\,, \,so that
\,$|\arg(-\om^m\<s)|<\pi$\,.
Inequalities \Ref{ineq} give us relations between \,$m$ \,and \,$\phi$\,,
which are exactly the inequalities in \Ref{Min}. We also have
\begin{align*}
& \Ups(t_m,\yy,\zz)\,=\,n\>\om^m\<s\>+
\sum_{a=1}^n\>z_a \log(-\om^m\<s)\>+\>
\frac n2\log\Bigl(\frac{2\pi}{-\om^m\<s}\Bigr)+
\>\sum_{j=1}^{n-1}\,\log(y_j\<-\om^m\<s)\>+\,\Oc(1/r)\,,
\\
& \frac{d^2}{dt^2}\>\Ups(t_m,\yy,\zz)\,=\,
\frac{n}{-\om^m\<s}\,+\,\Oc(1/r^2)\,.
\end{align*}
We apply the steepest descent method to the integral in \Ref{Pm} as in
\cite[Appendix 1]{Gu} and obtain
\begin{align*}
\Psi^m(s^n,\yy, \zz)\,&{}=\,
\sqrt{\frac{-\om^m\<s}{2\pi n}}\;e^{n\>\om^m\<s}\>
(-\om^m\<s)^{\sum_{a=1}^n z_a-n/2}\,(2\pi)^{n/2}\,
\prod_{j=1}^{n-1}\,(y_j\<-\om^m\<s)\,\bigl(1+\Oc(1/r)\bigr)
\\
&{}=\,\frac{(2\pi)^{(n-1)/2}}{\sqrt n}\;e^{n\>\om^m\<s}\>
(-\om^m\<s)^{\sum_{a=1}^n z_a +(1-n)/2}\,
\prod_{j=1}^{n-1}\,(y_j\<-\om^m\<s)\,\bigl(1+\Oc(1/s)\bigr)\>,
\end{align*}
which proves the lemma.
\end{proof}

\subsection{Admissible $\phi$ and $m$}
\label{sec adm}

\begin{cor}
\label{cor phm}
If the argument \,$\phi$ of \,$s$ satisfies the inequalities
\beq
\label{ad phm}
\frac{k}{n}\,<\,\phi\,<\,\frac{k+1}n\;, \qquad \on{for\ some}\ \ k\in\Z\,,
\eeq
then there are exactly \,$n$ integers $m$ satisfying \Ref{Min}.
They are \,$m=k+1\lc k+n$. Hence each element of the basis
$\{\Psi^{k+m}(s^n,\yy, \zz)\;|\;m=1\lc n\}$ of the space of solutions of
the joint system of equations \Ref{qde} and \Ref{qKZ} has the asymptotic
expansion \Ref{exp}.
\qed
\end{cor}

\begin{cor}
\label{cor mres}
If \,$\phi=k/n$ \,for some \,$k\in\Z$\,, then there are exactly
\,$n-1$ integers \,$m$ satisfying \Ref{Min}. They are \,$m=k+1\lc k+n-1$\,.
\qed
\end{cor}

We say that $\phi\in\R$ is {\it resonant}, if $\phi=k/n$ for some $k\in\Z$.

\begin{cor}
\label{cor mph}
Given \,$m\in \Z$\,, the function \,$\Psi^{m}(s^n,\yy, \zz)$ has the asymptotic
expansion \Ref{exp} if the argument \,$\phi$ of \,$s$ satisfies
the inequalities
\beq
\label{m phe}
\frac mn - 1\,<\,\phi\,<\,\frac mn\;,
\eeq
cf.~\Ref{Min}. Thus, the function \,$\Psi^{m}(s^n,\yy, \zz)$ has the asymptotic
expansion \Ref{exp} on \,$\C$ with the ray \,$\phi= m/n$ deleted and
the argument of \,$s$ fixed by \Ref{m phe}.
\qed

\end{cor}

\subsection{Stokes rays}

The
{\it Stokes rays} in $\C$ with coordinate $s=re^{-2\pii\phi}$ are the rays defined by the equations
\bean
\label{Str}
\phi \ = \ \frac r{2n},
\qquad r\in \Z.
\eean
The rays with $r$ even (resp., odd) will be called {\it even} (resp., {\it odd}).

\vsk.2>

Consider an interval $k/n<\phi <(k+1)/n$
between consecutive even rays.
Then each element of the basis
$\{ \Psi^{k+m}(s^n,\yy, \zz)\, | \, m=1\lc n\}$ has the asymptotic expansion \Ref{exp}
on that interval,
see Corollary \ref{cor phm}.

\vsk.2>

For given $k/n< \phi <(k+1)/n$
and $r\to\infty$, the absolute value of a basis solution
$ \Psi^{k+m}(s^n,\yy, \zz)$ is determined by the real number $\Re (\om^{k+m} s)$\,.
Namely, if $\Re (\om^{k+m_1} s) < \Re (\om^{k+m_2} s)$ for some $1\leq m_1, m_2\leq n$,
then
\be
|\Psi^{k+m_1}(s^n,\yy, \zz)| \ll |\Psi^{k+m_2}(s^n,\yy, \zz)|
\qquad \on{as}\ \ r\to \8\,,
\ee
see formula \Ref{exp}.

\goodbreak
\vsk.2>
The meaning of Stokes rays is explained by the following lemma.

\begin{lem}
\label{lem rays}
A number \,$\phi\in\R$ is of the form \,$\phi = r/2n$ for some \,$r\in\Z$,
if and only if there are \,$m_1, m_2$ such that
\,$\Re (\om^{m_1} s) = \Re(\om^{m_2} s)$ and
$\;m_1\not \equiv m_2\ \,(\on{mod}\;n)$\,.
\qed
\end{lem}

\subsection{Definition of Stokes bases}

\begin{defn}
Let $\{I_1(s^n,\yy,\zz), \dots,I_{n}(s^n,\yy,\zz)\}$ be a basis of solutions of the joint system of equations
\Ref{qde} and \Ref{qKZ}. We say that the basis is a {\it Stokes basis} on an interval $(a,b)$
\vvgood
if the basis can be reordered
so that for every $m=1\lc n$ and every non-resonant $\psi\in(a,b)$, we have
\beq
\label{st b}
I_m(s^n,\yy, \zz)\,=\,
\frac{(2\pi)^{(n-1)/2}}{\sqrt n}\;e^{n\>\om^m\<s}\>
(-\om^m\<s)^{\sum_{a=1}^n z_a +(1-n)/2}\,
\prod_{j=1}^{n-1}\,(y_j\<-\om^m\<s)\,\bigl(1+\Oc(1/s)\bigr)\kern-.4em
\eeq
as \,$s\to\8$\,.
\end{defn}

\vsk.3>
For example, for $k\in \Z$, the basis $ \{\Psi^{k+m}(s^n,\yy, \zz)\, | \, m=1\lc n\}$ is a Stokes basis on
the interval $(k/n,(k+1)/n)$.

\vsk.3>

For any ray $\phi = a$, which is not a Stokes ray, we will construct a Stokes basis on the interval
$(a-1/2-\epe, a +\epe)$,
where $\epe$ is a small positive number. We will
formulate the result in terms of a suitable braid group action.

\section{Exceptional bases and braid group}
\label{sec ebbg}

\subsection{Braid group action}

Let $M_n$ be a free $\C[\bs Z^{\pm1}]$-module
with basis $\{e_1\lc e_n\}$.
Define a sesquilinear form $A$ on $M_n$ by the formulas:
\bea
A(e_i,e_i) =1, \quad
A(e_i,e_j)=0 \ \on{for}\ i>j, \quad
A(e_i,e_j) = m_{j-i}(\ZZ)\ \on {for }\ i<j,
\\
A(a(\bs Z)x, b(\bs Z)y)\ =\ a(\bs Z^{-1})b(\bs Z)\,A(x,y) \quad \on{for} a,b \in \C[\bs Z^{\pm1}], \ x,y\in M_n.
\eea
Here the elements $m_k(\ZZ)\in\C[\ZZ^{\pm1}]$ for $k\in\Z_{\geq 0}$ are defined in \Ref{sm}.
Cf. Section \ref{sec:equK}.

\vsk.2>
The matrix of $A$ in the basis $\{e_1\lc e_n\}$
will be called {\it canonical}.

\vsk.2>

A basis $\{v_1\lc v_n\}$ of $M_n$ will be called {\it exceptional} if
\bea
A(v_i,v_i) =1, \quad
A(v_i,v_j)=0 \ \on{for}\ i>j.
\eea
In particular the basis $\{e_1\lc e_n\}$ is exceptional.

\vsk.2>

Let $\mc B_n$ be the braid group on $n$ strands with standard generators $\tau_1\lc\tau_{n-1}$.
The element
\bean
\label{Co}
C=\tau_1\tau_2\dots\tau_{n-1} \in \mc B_n
\eean
is called the {\it Coxeter element}.

\begin{lem}

The braid group acts on the set of exceptional bases
by the formula,
\be
Q\,=\,\{v_1\lc v_n\}\ \mapsto \
\tau_i Q\,=\,\{\dots, v_{i-1}, v_{i+1} -A(v_i,v_{i+1})v_i, v_i,v_{i+2},\dots\}.
\ee
\end{lem}

\begin{proof}
The fact that the basis \,$\tau_i Q$ is exceptional, if \,$Q$ is exceptional,
and the equality \,$\tau_{i}\tau_{i+1}\tau_{i}Q=\tau_{i}\tau_{i+1}\tau_{i}Q$
\,are checked by direct calculations.
\end{proof}

\begin{lem}
\label{lem CQ}
Let $Q=\{v_1\lc v_n\}$ be an exceptional basis
in which the matrix of $A$ is canonical.
Then
\bean
\label{Co el}
C Q=\{
v_n-s_1(\bs Z)v_{n-1}+ \dots +(-1)^ns_{n}(\bs Z)v_1, v_1,v_{2}\lc v_{n-1}\}.
\eean
Moreover, if we multiply the first element of the basis $CQ$ by
$(-1)^{n+1}s_n(\bs Z^{-1})$, then the basis will remain exceptional and
the matrix of $A$ in this new basis
\bean
\label{Co ba}
\{(-1)^{n+1}s_n(\bs Z^{-1})(v_n-s_1(\bs Z)v_{n-1}+ \dots +s_{n}(\bs Z)v_1), v_1,v_{2}\lc v_{n-1}\}
\eean
is canonical.

\end{lem}

\begin{proof}
By induction we observe that
\bea
\tau_i\tau_{i+1}\dots\tau_{n-1} Q=
(v_1\lc v_{i-1}, v_n-s_1(\bs Z)v_{n-1}+ \dots +s_{n-i}(\bs Z) v_i, v_i,v_{i+1}\lc v_n\}.
\eea
Then we calculate the matrix of $A$ relative to
the basis $C'Q$ from the definitions. In these calculations we use
relations \Ref{ms}.
\end{proof}

The map of bases
\bea
\{v_1\lc v_n\} \mapsto
\{(-1)^{n+1}s_n(\bs Z^{-1})(v_n-s_1(\bs Z)v_{n-1}+ \dots +s_{n}(\bs Z)v_1), v_1,v_{2}\lc v_{n-1}\}
\eea
will be called the {\it modified Coxeter map} and denoted by $C'$.

\subsection{The element $\ga_n\in\mc B_n$}

Let $\ell=n-1$ for $n$ odd and $\ell=n-2$ for $n$ even.
Thus $\ell$ is always even. Set \,$\ga_2=1$\,, and for \,$n\ge 3$\,,
\be
\beta_k = \tau_k \tau_{k+1} \dots \tau_{n-1},
\qquad
\gamma_n = \beta_\ell \beta_{\ell-2} \dots \beta_2.
\ee
For example,
\vvn->
\begin{align*}
\ga_3 & {}=\tau_2\,,
\\
\ga_4 & {}=\tau_2\tau_3\,,
\\
\ga_5 & {}=(\tau_4)\>(\tau_2\tau_3\tau_4)\,,
\\
\ga_6 & {}=(\tau_4\tau_5)\>(\tau_2\tau_3\tau_4\tau_5)\,.
\end{align*}
Define
\vvn-.8>
\begin{alignat*}3
\dl_{n,\on{odd}} &{}=\tau_1\tau_3\dots\tau_{n-2}\,, &
\dl_{n,\on{even}} &{}=\tau_2\tau_4\dots\tau_{n-1}\,,\qquad &&
\textrm{for \,$n$ \,odd}\,,
\\
\dl_{n,\on{odd}} &{}=\tau_1\tau_3\dots\tau_{n-1}\,,\qquad &
\dl_{n,\on{even}} &{}=\tau_2\tau_4\dots\tau_{n-2}\,, &&
\textrm{for \,$n$ \,even}\,.
\end{alignat*}

\begin{lem}
\label{lem ga}
We have the following identity in \,$\mc B_n$\,:
\beq
\label{od ev}
\dl_{n,\on{even}} \,\dl_{n,\on{odd}} \,\ga_n\,=\,\ga_n\, C\,.
\eeq

\end{lem}

\begin{proof} The proof is straightforward.
\end{proof}

\subsection{Bases $Q'$ and $Q''$}
\label{sec QQ}

Let $n=2k+1$.
Let $Q=\{v_1\lc v_n\}$ be a basis of $M_n$.
For $1\le l\leq m\leq n$ denote
\bean
\label{vml}
v_{m}(l) = v_m-s_1(\ZZ)v_{m-1} +\dots+(-1)^{m-l}s_{m-l}(\ZZ)v_l.
\eean

\vsk.2>

Introduce a basis \,$Q'$ in which the vectors $v_1\lc v_{k+1}$ stay at the positions
1, 3, 5, \dots, $2k+1$, respectively, and the vectors $v_{2k+1}(2)$,
$v_{2k}(3)$, \dots, $v_{k+2}(k+1)$ stay at the positions
2, 4, 6, \dots, $2k$, respectively.

Introduce a basis \,$Q''$ in which the vectors $v_1\lc v_{k+1}$
stay at the positions
2, 4, 6, \dots, $2k$, $2k+1$, respectively, and the vectors $v_{2k+1}(1)$,
$v_{2k}(2)$, \dots, $v_{k+2}(k+1)$ stay at the positions
1, 3, 5, \dots, $2k-1$, respectively.

For example for $n=5$, we have
\begin{align}
\label{Q5}
Q'=\,\{ {}&
v_1, v_5-s_1(\ZZ)v_4+s_2(\ZZ)v_3-s_3(\ZZ)v_2, v_2, v_4-s_1(\ZZ)v_3,v_3\}\,,
\\[4pt]
\notag
Q''=\,\{ {}&
v_5-s_1(\ZZ)v_4+s_2(\ZZ)v_3-s_3(\ZZ)v_2+s_4(\ZZ)v_1,v_1,
\\[1pt]
\notag
& v_4-s_1(\ZZ)v_3 +s_2(\ZZ)v_2,v_2, v_3 \}\,.
\end{align}

Let $n=2k$.
Let $Q=\{v_1\lc v_n\}$ be a basis of $M_n$.
Introduce a basis $Q'$, in which the vectors $v_1\lc v_{k+1}$ stay at the positions
1,3, 5, \dots, $2k-1$, $2k$, respectively, and the vectors $v_{2k}(2)$,
$v_{2k-1}(3)$, \dots, $v_{k+2}(k)$ stay at the positions
2, 4, 6, \dots, $2k-2$, respectively.

Introduce a basis $Q''$, in which the vectors $v_1\lc v_k$ stay at the positions
2, 4, 6, \dots, $2k$, respectively, and the vectors $v_{2k}(1)$,
$v_{2k-1}(2)$, \dots, $v_{k+1}(k)$ stay at the positions
1, 3, 5, \dots, $2k-1$, respectively.

For example for $n=6$, we have
\begin{align}
\label{Q4}
Q'=\,\{ {}&
v_1, v_6-s_1(\ZZ)v_5+s_2(\ZZ)v_4-s_3(\ZZ)v_3+s_4(\ZZ)v_2, v_2,
\\
\notag
& v_5-s_1(\ZZ)v_4+s_2(\ZZ)v_3, v_3, v_4\},
\\[4pt]
\notag
Q''=\,\{ {}&
v_6-s_1(\ZZ)v_5+s_2(\ZZ)v_4-s_3(\ZZ)v_3 + s_4(\ZZ)v_2-s_5(\ZZ)v_1,
\\
\notag
& v_1, v_5-s_1(\ZZ)v_4 +s_2(\ZZ)v_3-s_3(\ZZ)v_2, v_2, v_4-s_1(\ZZ)v_3,v_3\}.
\end{align}

\begin{lem}
\label{lem gabo}
Let \,$n>1$. Let \,$Q=\{v_1\lc v_{n}\}$ be a basis of $M_n$ such that
the matrix of $A$ relative to $Q$ is canonical. Then
\be
\ga_nQ\,=\,Q', \qquad \dl_{n,\on{odd}} Q'=\,Q''.
\ee
\end{lem}

\begin{proof} The proof is straightforward.
\end{proof}

\subsection{Modules $M_n$, $ K_{T^n}(P^{n-1},\C)$, and $\mc S_n$}

\begin{lem}
\label{lem MK}
The map
$\iota : M_n \to K_{T^n}(P^{n-1},\C)$, defined by
\bean
\label{MK}
\iota\ :\
e_j \mapsto X^{j-1},
\qquad j=1\lc n,
\eean
is an isomorphism of $\C[\ZZ^{\pm1}]$-modules,
which identifies the form $A$
on $M_n$ with the form $A$ on $K_{T^n}(P^{n-1},\C)$.
\qed

\end{lem}

\vsk.2>
Recall the isomorphism $\theta : K_{T^n}(P^{n-1},\C) \to \mc S_n$. The composition
isomorphism
$\theta\circ\iota : M_n\to \mc S_n$ is defined by
\bean
\label{MS}
\theta\circ\iota \ :\ e_m\ \mapsto \ \Psi^m, \qquad m=1\lc n.
\eean

Using the isomorphism $\theta\circ\iota $ we define
exceptional bases of $\mc S_n$ with the action of the braid group $\mc B_n$
on them.

\subsection{Exceptional bases of $\mc S_n$}

\begin{lem}
\label{lem eS}
For every $k\in \Z$, the basis $Q_k = \{\Psi^{k+1}\lc\Psi^{k+n}\}$ of $\mc S_n$
is an exceptional basis,
in which the matrix of $A$ is canonical.
We also have $C' Q_k=Q_{k-1}$.
\qed
\end{lem}

\begin{proof} The first statement follows from Lemma \ref{lem Kfo}.
The second statement follows from Lemma \ref{lem CQ} and formula \Ref{Psi reln}.
\end{proof}

Using the formulas of Section \ref{sec QQ} we assign to every basis $Q_k$
two exceptional bases $Q_k'$ and $Q''_k$.
For example for $n=5$, we define
\begin{align*}
Q_k'=\{ {}&
\Psi^{k+1}, \Psi^{k+5}-s_1(\ZZ)\Psi^{k+4}+s_2(\ZZ)\Psi^{k+3}-s_3(\ZZ)\Psi^{k+2},
\Psi^{k+2},
\\
& \Psi^{k+4}-s_1(\ZZ)\Psi^{k+3},\Psi^{k+3}\},
\\
\notag
Q''_k=\{ {}&
\Psi^{k+5}-s_1(\ZZ)\Psi^{k+4}+s_2(\ZZ)\Psi^{k+3}-s_3(\ZZ)\Psi^{k+2}+s_4(\ZZ)\Psi^{k+1},
\Psi^{k+1},
\\
\notag
&  \Psi^{k+4}-s_1(\ZZ)\Psi^{k+3}
+s_2(\ZZ)\Psi^{k+2},\Psi^{k+2}, \Psi^{k+3} \}.
\end{align*}

\goodbreak
\noindent
cf.~\Ref{Q5}. For any $n$ and $k$ we have
\beq
\label{gaQ}
\ga_nQ_k = Q_k', \qquad \dl_{n,\on{odd}}Q_k'= Q''_k\,,
\eeq
by Lemma \ref{lem gabo}.

\begin{lem}
\label{lem per}
For any \,$n$ and \,$k\in\Z$, multiplying the first basis vector of the basis
\,$\dl_{n,\on{even}}Q''_k$ by \,$(-1)^{n+1}s_n(\bs Z^{-1})$ \,yields
the basis \,$Q_{k-1}'$.
\end{lem}

\begin{proof}
The lemma follows from Lemmas \ref{lem ga} and \ref{lem eS}.
\end{proof}

\section{Stokes bases }
\label{sec StB}

\subsection{Main theorem}

\begin{thm}
\label{thm main}
The basis $Q_k'$ is a Stokes basis on the interval
$(a-1/2-\epe, a +\epe)$ if $a\in ((2k+1)/2n, (k+1)/n)$ and $\epe>0$ is small enough.
The basis $Q''_k$ is a Stokes basis on the interval
$(a-1/2-\epe, a +\epe)$ if $a\in (k/n, (2k+1)/2n)$ and $\epe>0$ is small enough.
\end{thm}

The smallness of $\epe$ means that
the intervals
$(a, a+\epe)$ and $(a-1/2-\epe, a-1/2)$ do not contain points of the form $r/2n$ where $r\in\Z$.

\begin{cor}
\label{cor tab}
Consider the three consecutive asymptotic bases $Q'_k, Q''_k, Q'_{k-1}$.
Then
$Q''_k = \dl_{n,\on{odd}}Q'_k$, and $ Q'_{k-1}$ is obtained from the basis
$\dl_{n,\on{even}}Q''_k$ by multiplying the first basis vector
of $\dl_{n,\on{even}}Q''_k$ by $(-1)^{n+1}s_n(\bs Z^{-1})$.
\qed

\end{cor}

It is enough to prove Theorem \ref{thm main} for $k=0$, since the case of arbitrary $k$ is obtained
from the case of $k=0$ by the change of variables
$m\mapsto k+m$ and $\phi \mapsto \phi + k/n$
in the integral \Ref{Pm}.
Theorem \ref{thm main} for $k=0$ is proved in Section \ref{sec proo}.

\subsection{Paths and functions}
\label{sec path}

For integers $l\leq m$ we define the {\it path} $C^m(l)$ on the regular $n$-gone $\Dl$
with vertices
$\{\om^1, \om^2\lc\om^n\}$ as the path along the boundary of $\Dl$,
which starts at the vertex $\om^l$ and goes to the vertex
$\om^m$ through the vertices $\om^{l+1}, \dots\om^{m-1}$. The vertices $\om^m$ and $ \om^l$
are the {\it head} and {\it tail} of the path. The number $m-l$ is the
{\it length} of the path. The path $C^m(l)$ goes around $\Dl$ counterclockwise.

\vsk.2>
All our paths will be of length less than $n$.

\vsk.2>

Let $l\leq m$ and $m-l<n$. Define the
{\it reflected path} $\bar C^m(l)$ to be the path along the boundary of $\Dl$,
which goes from
the vertex $\om^{l-1}$ to the vertex $\om^{m-n}=\om^m$ through the vertices
$\om^{l-2}, \om^{l-3}$, \dots, $\om^{m-n+1}$.
The reflected path $\bar C^m(l)$ goes around $\Dl$ clockwise.

Both $C^m(l)$ and $\bar C^m(l)$ have the same heads. The sum of lengths of
$C^m(l)$ and $\bar C^m(l)$ equals $n-1$.

\begin{defn}
\label{def 1}
Let $l\leq m$ and $m-l<n$. Assign to the path $C^m(l)$ the function
\beq
\label{Pml2}
\Psi^{m}(l)\,=\,
\Psi^m-s_1(\ZZ)\Psi^{m-1} +\dots+(-1)^{m-l}s_{m-l}(\ZZ)\Psi^l,
\eeq
and to the reflected path $\bar C^m(l)$ the function
\beq
\label{Pmll}
\bar \Psi^{m}(l)\,=\,
(-1)^{n-1} s_{n}(\ZZ)\Psi^{m-n}+(-1)^{n-2}s_{n-1}(\ZZ)\Psi^{m-n+1} +\ldots
+(-1)^{m-l}s_{m-l+1}(\ZZ)\Psi^{l-1}.\kern-1em
\eeq

\end{defn}

Notice that
the functions $\Psi^{m}(l)$ and $\bar \Psi^{m}(l)$ are equal by formula \Ref{Psi reln},
while the summands in
$\Psi^m(l)$ correspond to the vertices of the path $C^m(l)$ and the
summands in $\bar\Psi^m(l)$ correspond to the vertices of the path
$\bar C^m(l)$.

\vsk.2>

Consider the rotated $n$-gone $e^{-2\pii \phi}\Dl$ and rotated
paths $e^{-2\pii \phi}C^m(l)$, $e^{-2\pii \phi}\bar C^m(l)$. We say that the path
$e^{-2\pii \phi}C^m(l)$ is {\it admissible}
if
the number $\Re (e^{-2\pii\phi}\om^m)$ is greater than
the number $\Re (e^{-2\pii\phi}\om^k)$ for any other vertex of the path $C^m(l)$,
and we say that the path
$e^{-2\pii \phi}\bar C^m(l)$ is {\it admissible}
if
the number $\Re (e^{-2\pii\phi}\om^m)$ is greater than
the number $\Re (e^{-2\pii\phi}\om^k)$ for any other vertex of the path $\bar C^m(l)$.

\subsection{Bases $Q_0'$, $Q''_0$}

We have
\begin{align}
\label{Qp0}
Q_0'\,=\,\{ {}&
\Psi^{1}, \Psi^n-s_1(\ZZ)\Psi^{n-1}+\dots+ (-1)^{n-2}s_{n-2}(\ZZ) \Psi^2,
\\
\notag
& \Psi^2, \Psi^{n-1}-s_1(\ZZ)\Psi^{n-2}+\dots+ (-1)^{n-4}s_{n-4}(\ZZ) \Psi^3,
\Psi^3, \dots\}\,,
\\[5pt]
\label{Qpp0}
Q''_0\,=\, \{ {}&
\Psi^n-s_1(\ZZ)\Psi^{n-1}+\dots+ (-1)^{n-1}s_{n-1}(\ZZ) \Psi^1, \Psi^1,
\\
\notag
& \Psi^{n-1}-s_1(\ZZ)\Psi^{n-2}+\dots+ (-1)^{n-2}s_{n-3}(\ZZ) \Psi^2, \Psi^2\lc\}.
\end{align}

\subsection{Proof of Theorem \ref{thm main} for $k=0$}
\label{sec proo}

We will prove the theorem for $Q_0'$. The proof for $Q''_0$ is completely similar.

\vsk.2>
We will prove that the basis $Q_0'$ is a Stokes basis on the interval $(a-1/2-\epe, a+\epe)$,
if $a\in (1/2n, 1/n)$, where $\epe$ is small.
The Stokes rays divide the non-resonant points of the interval
$(a-1/2-\epe, a+\epe)$ into the subintervals $(1/2n, a+\epe)$,
$(0, 1/2n)$, $ (-1/2n,0),\dots$. The first and last of these subintervals are shorter than the intervals
between the Stokes rays, since they have boundary points $a+\epe$, $a-1/2-\epe$ lying in between Stokes rays.

We will prove that $Q_0'$ is a Stokes basis on each of these subintervals.

\vsk.2>
We start with the first two subinterval $(1/2n,a+\epe)$ and $(0, 1/2n)$.
We assume that $\epe$ is small so that $(1/2n,a+\epe)\subset (1/2n,1/n)$.

\vsk.2>
The functions $\Psi^m(s^n,\yy,\zz)$, $m=1\lc n$, appearing in \Ref{Qp0} are all admissible for the interval
$(0,1/n)$ in the sense of Corollary \ref{cor phm}.
For $\phi\in (0,1/n)$ each of these functions has an asymptotic expansion
with the leading term $\exp(n re^{-2\pii\phi}\om^m)$. The magnitude of a function
$\Psi^m(s^n,\yy,\zz)$
is determined by the real part of the
number $e^{-2\pii\phi}\om^m$.

Hence to order the magnitudes of the solutions $\Psi^m(s^n,\yy,\zz)$, $m=1\lc n$, for $\phi\in (0,1/n)$
we need to consider the rotated $n$-gone $e^{-2\pii\phi}\Dl$ and order the real
parts of its vertices.

\vsk.2>

Using notations of Section \ref{sec path} we write
\beq
\label{Op}
Q_0'=\{\Psi^1(1),\Psi^n(2), \Psi^2(2),\Psi^{n-1}(3),\Psi^3(3), \dots\}\,.
\vvgood
\eeq
These functions are
the functions, which were assigned to the sequence of paths
$\{C^1(1)$, $C^n(2)$, $C^2(2)$, $C^{n-1}(3)$, $C^3(3)$, $\dots \}$ in Definition \ref{def 1}.
Each of these paths is admissible with respect to
$e^{-2\pii\phi}\Dl$ for $\phi\in(0,1/n)$. Hence
each linear combination $\Psi^m(l)$ appearing in this sequence has asymptotic expansion with leading term
$\exp(n re^{-2\pii\phi}\om^m)$, coming from the summand
$\Psi^m$ of $\Psi^m(l)$, corresponding to the head of the path $C^m(l)$. Therefore
the basis $Q_0'$ is an asymptotic basis on
the two subintervals $(1/2n,a+\epe)$ and $(0, 1/2n)$.

\vsk.3>

Consider the next two subintervals $(-1/2n,0)$ and $(-1/n, -1/2n)$. On the interval
\\
$(-1/n, 0)$ the admissible functions are $\Psi^0, \dots,\Psi^{n-1}$.
For $\phi\in (-1/n,0)$ each of these functions has an asymptotic expansion
with the leading term $\exp(n re^{-2\pii\phi}\om^m)$.

\vsk.2>

In formula \Ref{Qp0} the function $\Psi^n(2)$ is the only function that
uses the non-admissible function $\Psi^n$. We replace
the presentation of $\Psi^n(2)$ in \Ref{Qp0} by the equal sum
\be
\bar\Psi^n(2) = (-1)^{n-1} s_{n}(\ZZ)\Psi^{0}+(-1)^{n-2}s_{n-1}(\ZZ)\Psi^{1},
\ee
which uses only the admissible functions $\Psi^0, \dots,\Psi^{n-1}$. On the interval
$(-1/n,0)$ we have
\bean
\label{new p}
Q_0'=\{\Psi^1(1), \bar\Psi^n(2),\Psi^2(2), \Psi^{n-1}(3),\Psi^3(3),
\dots\},
\eean
where the dots indicates the same functions as
in \Ref{Qp0}. This new presentation of the basis $Q_0'$ corresponds to the sequence
of paths $\{C^1(1),\bar C^n(2), C^2(2), C^{n-1}(3)$, $C^3(3)$, $\dots \}$.

Each of these paths is admissible with respect to
$e^{-2\pii\phi}\Dl$ for $\phi\in(-1/n,0)$. Hence
each linear combination of the functions $\Psi^0, \dots,\Psi^{n-1}$
appearing as a basis vector in \Ref{new p}
has asymptotic expansion with leading term
$\exp(n re^{-2\pii\phi}\om^m)$, coming from the summand
$\Psi^m$ corresponding to the head of the corresponding path. Therefore
the basis $Q_0'$ is an asymptotic basis on the two subintervals
$(-1/2n,0)$ and $(-1/n, -1/2n)$.

\vsk.2>

On the next two subintervals $(-3/2n,-1/n)$ and $(-2/n, -3/2n)$
the admissible functions are $\Psi^{-1}, \dots,\Psi^{n-2}$.
For $\phi\in (-2/n,-1/n)$ each of these functions has an asymptotic expansion
with the leading term $\exp(n re^{-2\pii\phi}\om^m)$.

In formula \Ref{new p} the function $\Psi^{n-1}(3)$ is the only function that
uses the non-admissible $\Psi^{n-1}$. We replace
the presentation of $\Psi^{n-1}(3)$ in \Ref{Qp0} by the equal sum
\begin{align*}
\bar\Psi^{n-1}(3)\,={} & \,
(-1)^{n-1} s_{n}(\ZZ)\Psi^{-1}+(-1)^{n-2}s_{n-1}(\ZZ)\Psi^{0}
\\[1pt]
&{}+(-1)^{n-3}s_{n-2}(\ZZ)\Psi^{1}+(-1)^{n-4}s_{n-3}(\ZZ)\Psi^{2},
\end{align*}
which uses only the admissible functions $\Psi^{-1}, \dots,\Psi^{n-2}$. On the interval
$(-2/n,-1/n)$ we have
\bean
\label{new pp}
Q_0'=\{\Psi^1(1), \bar\Psi^n(2),\Psi^2(2), \bar \Psi^{n-1}(3),\Psi^3(3),
\dots\},
\eean
where the dots indicates the same functions as
in \Ref{Qp0}. This new presentation of the basis $Q_0'$ corresponds to the sequence
of paths $\{C^1(1),\bar C^n(2), C^2(2), \bar C^{n-1}(3)$, $C^3(3)$, $\dots \}$.

Each of these paths is admissible with respect to
$e^{-2\pii\phi}\Dl$ for $\phi\in(-2/n,-1/n)$. Hence
each linear combination of the functions $\Psi^{-1}, \dots,\Psi^{n-2}$
appearing as a basis vector in \Ref{new pp}
has asymptotic expansion with leading term
$\exp(n re^{-2\pii\phi}\om^m)$, coming from the summand
$\Psi^m$ corresponding to the head of the corresponding path. Therefore
the basis $Q_0'$ is an asymptotic basis on the two subintervals
$(-3/2n,-1/n)$ and $(-2/n, -3/2n)$.

Repeating this procedure we prove Theorem \ref{thm main} for $Q_0'$. See a similar reasoning in
\cite{Gu}.

\end{document}